\documentclass[a4paper,11pt]{amsart}

\usepackage[a4paper,hmargin={2.5cm,2.5cm},vmargin={2.5cm,2.5cm}]{geometry}

\usepackage{color}

\usepackage{stmaryrd}
\usepackage[leqno]{amsmath}
\usepackage{amssymb,amsthm}
\usepackage[mathscr]{euscript}
\usepackage{bbm}
\usepackage{dsfont}
\usepackage[all,arc]{xy}
\DeclareMathAlphabet{\mathpzc}{OT1}{pzc}{m}{it}

\theoremstyle{plain}
\newtheorem{theorem}{Theorem}[section]
\newtheorem{lemma}[theorem]{Lemma}
\newtheorem{proposition}[theorem]{Proposition}
\newtheorem{corollary}[theorem]{Corollary}

\theoremstyle{definition}

\newtheorem{examples}[theorem]{Examples}

\theoremstyle{remark}
\newtheorem{remark}[theorem]{Remark}

\newenvironment{eqcond}{\begin{enumerate}}{\end{enumerate}}

\newcommand{\rw}{\rightarrow}

\newcommand{\Rw}{\Rightarrow}

\newcommand{\hrw}{\hookrightarrow}

\newcommand{\modto}{{\longrightarrow\hspace*{-2.8ex}{\circ}\hspace*{1.2ex}}}

\DeclareMathOperator{\yoneda}{\mathpzc{y}}

\DeclareMathOperator{\downc}{\downarrow\!}
\DeclareMathOperator{\Cauchy}{Cauchy}
\DeclareMathOperator{\ev}{ev}
\DeclareMathOperator{\und}{\&}

\newcommand{\mate}[1]{\,^\ulcorner\! #1^\urcorner}

\newcommand{\catfont}[1]{\mathsf{#1}}

\newcommand{\V}{\catfont{V}}
\newcommand{\W}{\catfont{W}}
\newcommand{\two}{\catfont{2}}

\newcommand{\quantale}{(\V,\otimes,k)}

\newcommand{\SET}{\catfont{Set}}

\newcommand{\ORD}{\catfont{Ord}}
\newcommand{\MET}{\catfont{Met}}
\newcommand{\PROBMET}{\catfont{ProbMet}}

\newcommand{\Mod}[1]{#1\text{-}\catfont{Mod}}
\newcommand{\Cat}[1]{#1\text{-}\catfont{Cat}}

\newcommand{\eps}{\varepsilon}
\newcommand{\op}{\mathrm{op}}

\newcommand{\true}{\mathsf{true}}
\newcommand{\false}{\mathsf{false}}
\newcommand{\trunminus}{\ominus}
\newcommand{\trundiv}{\varobslash}

\newcommand{\blackright}{\mbox{ $-\!\mbox{\footnotesize $\bullet$}$ }}
\newcommand{\blackleft}{\mbox{ $\mbox{\footnotesize $\bullet$}\!-$ }}

\newcommand{\field}[1]{\mathds{#1}}

\newcommand{\N}{\field{N}}

\begin{document}

\title{Probabilistic Metric Spaces as enriched categories}

\author{Dirk Hofmann}
\address{Center for Research and Development in Mathematics and Applications, Departament of Matematics, University of Aveiro, 3810-193 Aveiro, Portugal}
\email{dirk@ua.pt}

\author{Carla David Reis}
\address{Center for Research and Development in Mathematics and Applications, University of Aveiro, Portugal \&
Department of Science and Technology, College of Management and Technology of Oliveira do Hospital, Polytechnic Institute of Coimbra, 3400-124 Oliveira do Hospital, Portugal}
\email{carla.reis@estgoh.ipc.pt}

\thanks{The first author acknowledges partial financial assistance by FEDER funds through COMPETE -- Operational Programme Factors of Competitiveness (Programa Operacional Factores de Competitividade) and by Portuguese funds through the Center for Research and Development in Mathematics and Applications (University of Aveiro) and the Portuguese Foundation for Science and Technology (FCT -- Funda\c{c}\~ao para a Ci\^encia e a Tecnologia), within project PEst-C/MAT/UI4106/2011 with COMPETE number FCOMP-01-0124-FEDER-022690, and the project MONDRIAN (under the contract PTDC/EIA-CCO/108302/2008).
The second author acknowledges partial financial assistance by grant SFRH/PROTEC/49762/2009 (FCT -- Funda\c{c}\~ao para a Ci\^encia e Tecnologia) from the Portuguese government.}


\begin{abstract}
In this paper we investigate Cauchy completeness and exponentiablity for quantale enriched categories, paying particular attention to probabilistic metric spaces.
\end{abstract}

\maketitle

\section*{Introduction}

Lawvere's ground-breaking paper \cite{Law_MetLogClo} presenting generalised metric spaces as enriched categories has motivated much work on the reconciliation of order, metric and category theory. However, the theory of categories enriched in a symmetric monoidal closed category  \cite{EK_CloCat,Kel_EnrCat} can become quickly technically very demanding, which prompted many authors to restrict themselves to the case where the enrichment takes place in a quantale (i.e.\ a thin symmetric monoidal closed category) where ``all diagrams commute'' and therefore all coherence issues disappear. This way the employed categorical notions and techniques have a very elementary formulation, however the theory still includes many interesting examples such as ordered sets, metric spaces and probabilistic metric spaces. We refer the reader in particular to the work \cite{Kop88,Flagg_ComplCont,Flagg_QuantCont,FSW_QDT,FK_ContSp} of Flagg \textit{et al.}\ on continuity spaces and the work \cite{BBR_GenMet,Rut98a} of the Amsterdam research group at CWI. Another important source of motivation was for us the work \cite{GR02} on the completion of fuzzy metric spaces, due to the similarities between the notions of fuzzy metric spaces and probabilistic metric spaces. In this paper we contribute to this line of research and study Cauchy completeness and exponentiability in quantale-enriched categories. We interpret our results in probabilistic metric spaces seen as categories enriched in the quantale of distribution functions, and show that in many cases categorical and classical notions coincide.  

One amazing insight of \cite{Law_MetLogClo} is a characterisation of the notion of Cauchy completeness for metric spaces using adjoint modules, giving further evidence to MacLane's motto ``adjoints occur almost everywhere'' \cite{MacLane_WorkMath}. This result was further generalised in \cite{Flagg_ComplCont} to categories enriched in a \emph{value quantale}: for such categories, Cauchy completeness can be equivalently described via modules and via Cauchy nets. Using the conceptual power of adjunction, in the first part of this paper we show that many results linking adjoint modules and Cauchy sequences (resp.\ nets) are valid under much milder assumptions. In the second part we investigate function spaces of quantale-enriched categories, showing in particular that injective probabilistic metric spaces are exponentiable.

\section{A brief introduction to quantale-enriched categories}

\subsection{Quantales}
Throughout this paper we consider a \emph{quantale} $\V=\quantale$, by which we mean a complete ordered set $\V$ equipped with an associative and commutative binary operation $\otimes:\V\times\V\to\V$ with neutral element $k$ satisfying
\begin{equation*}
u\otimes\bigvee_{i\in I}v_i=\bigvee_{i\in I}(u\otimes v_i),
\end{equation*}
for all $u,v_i\in \V$ and $i\in I$. In other words, each function $u\otimes-:\V\to\V$ ($u\in\V$) preserves suprema and therefore has a right adjoint $\hom(u,-):\V\to\V$. Consequently, there is a map $\hom:\V\times\V\to\V$ such that, for all $u,v,w\in\V$,
\begin{equation*}
u\otimes v\le w \iff v\le\hom(u,w).
\end{equation*}
\begin{examples}
\begin{enumerate}
\item The two-element Boolean algebra $\two=\{\false,\true\}$ is a quantale with tensor $\otimes=\&$ and $k=\true$. More general, every frame is a quantale with $\otimes=\wedge$ and $k=\top$.
\item The real half-line $[0,\infty]$ ordered by the ``greater or equal'' relation $\geqslant$ is a quantale, with tensor $\otimes=+$ (with $x+\infty=\infty+x=\infty$) for all $x,y\in[0,\infty]$; then $k=0$ is obviously the neutral element for $\otimes=+$ and one has $\hom(x,y)=y\trunminus x:=\max\{y-x,0\}$ (with $y-\infty=0$ and $\infty-x=\infty$ for $x,y\in[0,\infty]$, $x\neq\infty$). Also note that $0$ is the top and $\infty$ is the bottom element with respect to this order, and we have $\bigvee=\inf$ and $\bigwedge=\sup$ in $[0,\infty]$.
\item The interval $[0,1]$ is a quantale, with the usual order $\le$ and $\otimes=\cdot$ multiplication. The right adjoint is given here by ``division'' $\hom(x,y)=y\trundiv x$ where $y\trundiv 0=1$ and $y\trundiv x=\min\{\frac{y}{x},1\}$ for $x\neq 0$.
\end{enumerate}
\end{examples}
Given also a quantale $\W=(\W,\oplus,l)$ and a monotone map $F:\V\to\W$, we call $F$ a \emph{morphism of quantales} whenever  
\begin{align*}
 \bigvee_{i\in I}F(v_i)&=F(\bigvee_{i\in I}v_i), &
 F(u)\oplus F(v)&=F(u\otimes v), &
 l&=F(k),
\end{align*}
for all $u,v,v_i\in\V$ and $i\in I$. It turns out that for many applications it is enough to have inequalities above; in this case we say that $F$ is a \emph{lax morphism of quantales}. That is, a lax morphism of quantales $F:\V\to\W$ only needs to satisfy, for all $u,v\in\V$,
\begin{align*}
 F(u)\oplus F(v)&\le F(u\otimes v), &
 l&\le F(k).
\end{align*}
Note that the inequality $\bigvee_{i\in I}F(v_i)\le F(\bigvee_{i\in I}v_i)$ follows from monotonicity of $F$.

\begin{examples}\label{exs:MorphQuant}
The map $I:\two\to[0,\infty]$ interpreting $\false$ as $\infty$ and $\true$ as $0$ is a morphism of quantales, and note that $I$ is indeed monotone since we consider the ``greater or equal'' relation $\geqslant$ on $[0,\infty]$. Furthermore, $I:\two\to[0,\infty]$ has a left and a right adjoint given by
\begin{align*}
 O:[0,\infty] &\to\two, & P:[0,\infty]&\to\two\\
  x &\mapsto
\begin{cases}
 \true & \text{if $x<\infty$}\\
 \false & \text{if $x=\infty$}
\end{cases}
&
x &\mapsto
\begin{cases}
 \true & \text{if $x=0$}\\
 \false & \text{if $x>0$}
\end{cases}
\end{align*}
respectively. Here $O:[0,\infty]\to\two$ is a morphism of quantales as well, but $P:[0,\infty]\to\two$ is only a lax morphism of quantales.
 
This construction can be generalised to an arbitrary quantale $\V$. Firstly, the map $I:\two\to \V$ interpreting $\false$ as $\bot$ and $\true$ as $k$ is a morphism of quantales. Since $I$ preserves suprema it has a right adjoint $P:\V\to\two$; and $I$ preserves the top element precisely if $k=\top$, and in this case $I$ preserves all infima and therefore has a left adjoint $O$. Furthermore, the left and the right adjoint of $I$ are given by
\begin{align*}
 O:\V &\to\two  &  P:\V&\to\two\\
  x &\mapsto
\begin{cases}
 \true & \text{if $x \neq \bot$}\\
 \false & \text{if $x=\bot$}
\end{cases}
&
x &\mapsto
\begin{cases}
 \true & \text{if $x\ge k$}\\
 \false & \text{else}
\end{cases}
\end{align*}
respectively. Being left adjoint, $O:\V\to\two$ preserves suprema, and one also verifies easily that $O(k)=\true$; but $O$ is in general not a lax morphism of quantales since one only has $O(u\otimes v)\le O(u)\,\&\, O(v)$, with equality if and only if $u\otimes v=\bot \Rightarrow u=\bot \text{ or } v=\bot$, for all $u,v\in\V$. Finally, the right adjoint $P:\V\to\two$ is a lax morphism of quantales.

The bijection
\[
 E:[0,\infty]\to [0,1],\, x\mapsto\exp(-x)
\]
(where $\exp(-\infty)=0$) is a morphism of quantales, and so is its inverse
\[
 L:[0,1]\to[0,\infty],\,x\mapsto -\ln(x)
\]
(where $-\ln(0)=\infty$).
\end{examples}

In the sequel we will occasionally assume that $\V$ is \emph{completely distributive} \cite{Raney_CD}. This amounts to saying that suprema commute with infima in our complete lattice $\V$, which can be expressed by saying that the monotone map $\bigvee:\two^{\V^\op}\to\V$ preserves infima. Since $\two^{\V^\op}$ is complete, preservation of infima is equivalent to the existence of a left adjoint $A:\V\to\two^{\V^\op}$ of $\bigvee$ which can be described as follows. For $u,x\in \V$, one says that $u$ is \emph{totally below} $x$, written as $u\ll x$, if, for every $S\subseteq \V$, $x\le\bigvee S$ entails $u\in\downc S$. Then $\V$ is completely distributive if and only if, for every $x\in\V$, $x=\bigvee\{u\in\V\mid u\ll x\}$; and in this case one has $A(x)=\{u\in\V\mid u\ll x\}$. Note that the totally-below relation $\ll$ is defined for every (complete) ordered set $X$, and enjoys the following properties (see \cite{Woo_OrdAdj}, for instance):
\begin{enumerate}
\item if $x\ll y$, then $x\le y$,
\item $x\le y\ll z$ implies $x\ll z$ and $x\ll y\le z$ implies $x\ll z$,
\item if $x\ll z$, then there exists some $y\in X$ with $x\ll y\ll z$.
\end{enumerate}

\subsection{$\V$-categories}

A \emph{$\V$-category} $X=(X,a)$ is a set $X$ together with a map $a:X\times X\to\V$ satisfying
\begin{align*}
 k\le a(x,x), &&  a(x,y)\otimes a(y,z)\le a(x,z)
\end{align*}
for all $x,y,z\in X$. A \emph{$\V$-functor} $f:(X,a)\to (Y,b)$ is a map $f:X\to Y$ such that
\[
 a(x,y)\le b(f(x),f(y))
\]
for all $x,y\in X$. We call a $\V$-functor $f:(X,a)\to (Y,b)$ \emph{fully faithful} if $a(x,y)= b(f(x),f(y))$ for all $x,y\in X$. The resulting category of $\V$-categories and $\V$-functors will be denoted by $\Cat{\V}$. We note that the quantale $\V$ gives rise to the $\V$-category $\V=(\V,\hom)$.

\begin{examples}
For $\V=\two$, a $\V$-category is just a set equipped with a reflexive and transitive relation, and a $\V$-functor is a monotone map. Hence, $\Cat{\V}$ is the category $\ORD$ of (pre)ordered sets and monotone maps. For $\V=[0,\infty]$, a $\V$-category structure is a distance function $a:X\times X\to[0,\infty]$ which satisfies the conditions
\begin{align*}
 0\geqslant a(x,x) &&\text{and}&& a(x,y)+a(y,z)\geqslant a(x,z),
\end{align*}
for all $x,y,z\in X$; and a $\V$-functor is a non-expansive map. Hence, $\Cat{\V}$ is the category $\MET$ of (pre)metric spaces and non-expansive maps. However, in the sequel we follow the nomenclature of \cite{Law_MetLogClo} and call the objects of $\MET$ simply metric spaces, then a ``classical'' metric space becomes a \emph{separated} ($d(x,y)=0=d(y,x)$ implies $x=y$),  \emph{symmetric} ($d(x,y)=d(y,x)$) and \emph{finitary} ($d(x,y)<\infty$) metric space. In a similar manner, we do not assume an ordered set to be anti-symmetric, and therefore we call the objects of $\ORD$ simply ordered sets. Consequently, many notions of order theory such as suprema or infima are only unique up to equivalence $\simeq$, where $x\simeq x'$ if $x\le x'$ and $x'\le x$. However, we stress that our quantale $\V$ (being part of the syntax) is assumed to be anti-symmetric.
\end{examples}

Every lax morphism of quantales $F:\V\to\W$ induces a functor $F:\Cat{\V}\to\Cat{\W}$ which sends a $\V$-category $X=(X,a)$ to the $\W$-category $FX$ with the same underlying set $X$ and with the $\W$-categorical structure given by the composite
\[
 X\times X\xrightarrow{\;a\;}\V\xrightarrow{\;F\;}\W;
\]
and a $\V$-functor $f:(X,a)\to(Y,b)$ is sent to the $\W$-functor $Ff:=f:(X,F\cdot a)\to(Y,F\cdot b)$. If the monotone map $F:\V\to\W$ happens to have an adjoint $G:\W\to\V$ which is also a lax morphism of quantales, then the induced functor $G:\Cat{\W}\to\Cat{\V}$ is adjoint to $F:\Cat{\V}\to\Cat{\W}$. In particular, when $F:\V\to\W$ and $G:\W\to\V$ are inverse to each other, then they induce an isomorphism between $\Cat{\V}$ and $\Cat{\W}$.

\begin{examples}
For the (lax) morphisms of quantales considered in Examples \ref{exs:MorphQuant}, the functor $I:\ORD\to\MET$ interprets an order relation $\le$ on a set $X$ as the metric
\[
 d(x,y)=
\begin{cases}
 0 & \text{if }x\le y,\\
 \infty & \text{else.}
\end{cases}
\]
The functor $I:\ORD\to\MET$ has a left adjoint $O:\MET\to\ORD$ which takes a metric $d$ on $X$ to the order relation
\[
x\le y \text{ whenever } d(x,y)<\infty,
\]
and a right adjoint $P:\MET\to\ORD$ which sends a metric $d$ on $X$ to the order relation
 \[
x\le y \text{ whenever } 0\geqslant d(x,y).
\]
Finally, we find that $\MET\simeq\Cat{[0,\infty]}$ and $\Cat{[0,1]}$ are isomorphic.
\end{examples}

To every $\V$-category $X=(X,a)$ one associates its \emph{dual} $X^\op=(X,a^\circ)$ where $a^\circ(x,y)=a(y,x)$, for all $x,y\in X$. Since a $\V$-functor $f:(X,a)\to(Y,b)$ can also be seen as a $\V$-functor of type $(X,a^\circ)\to(Y,b^\circ)$, we actually obtain a functor $(-)^\op:\Cat{\V}\to\Cat{\V}$. A $\V$-category $X=(X,a)$ is called \emph{symmetric} whenever $X=X^\op$, which amounts to saying that $a(x,y)=a(y,x)$ for all $x,y\in X$.

There is a canonical forgetful functor $\Cat{\V}\to\ORD$ which sends a $\V$-category $X=(X,a)$ to the ordered set $(X,\le)$ where
\[
 x\le x':\iff k\le a(x,x'), \qquad\qquad (x,x'\in X).
\]
A $\V$-category $X=(X,a)$ is called \emph{separated} if the underlying order is anti-symmetric, that is, if $x\simeq y$ implies $x=y$, for all $x,y\in X$. Note that for $\V=[0,\infty]$ the notions of symmetry and separatedness coincide with the usual ones for metric spaces.

The order relation on $\V$-categories can be extended point-wise to $\V$-functors $f,g:(X,a)\to (Y,b)$: one defines $f\le g$ whenever $f(x)\le g(x)$ for all $x\in X$; and composition from either side preserves this order. One fundamental consequence of the fact that $\Cat{\V}$ has ordered hom-sets is the possibility to talk about \emph{adjunction}. Here a pair of $\V$-functors $f:(X,a)\to (Y,b)$ and $g:(Y,b)\to(X,a)$ forms an adjunction $f\dashv g$ whenever $1_X\le g\cdot f$ and $f\cdot g\le 1_{Y}$. Equivalently, $f\dashv g$ if and only if, for all $x\in X$ and $y\in Y$,
\[
 b(f(x),y)=a(x,g(y));
\]
and the formula above explains why one calls $f$ left adjoint and $g$ right adjoint. We also recall that a left adjoint $f$ has at most one right adjoint since $f\dashv g$ and $f\dashv g'$ imply $g\simeq g'$; and dually, $f\dashv g$ and $f'\dashv g$ imply $f\simeq f'$.

The canonical forgetful functor $\Cat{\V}\to\SET,\,(X,a)\mapsto X$ is topological (see \cite{AHS}) where the initial structure on $X$ with respect to the family $f_i:X\to (X_i,a_i)$ ($i\in I$) is given by
\[
 a(x,x'):=\bigwedge_{i\in I}a_i(f_i(x),f_i(x')),
\]
for all $x,x'\in X$. Hence, $\Cat{\V}$ admits all limits and all colimits which are, moreover, preserved by $\Cat{\V}\to\SET$. In particular, the product $X\times Y$ of $\V$-categories $X=(X,a)$ and $Y=(Y,b)$ can be constructed by taking the Cartesian product $X\times Y$ of the sets $X$ and $Y$, and then turning this into a $\V$-category by putting
\[
 d((x,y),(x',y'))=a(x,x')\wedge b(y,y'),
\]
for all $(x,y),(x',y')\in X\times Y$.

More important to us is, however, the structure
\[
 (a\otimes b)((x,y),(x',y'))=a(x,x')\otimes b(y,y')
\]
on the set $X\times Y$, defining the $\V$-category $X\otimes Y=(X\times Y,a\otimes b)$. This tensor product $\otimes$ on $\Cat{\V}$ is associative and commutative, and has $E=(1,k)$ (with a singleton set $1=\{\star\}$ and $k(\star,\star)=k$) as neutral object. Note that in general $E=(1,k)$ must be distinguished from the terminal object $1=(1,\top)$ in $\Cat{\V}$. What makes this structure more interesting is the fact that, unlike $X\times-$, the functor $X\otimes-:\Cat{\V}\to\Cat{\V}$ has a right adjoint $[X,-]$, for every $\V$-category $X=(X,a)$. Here, given also $Y=(Y,b)$ in $\Cat{\V}$, $[X,Y]$ can be taken as $[X,Y]=(\Cat{\V}(X,Y),[a,b])$ where 
\[
 [a,b](f,g)=\bigwedge_{x\in X}b(f(x),g(x)).
\]
In the sequel we will pay particular attention to the $\V$-category $[X^\op,\V]$ where $\V=(\V,\hom)$. To simplify notation, we write $[h,h']$ instead of $[a^\circ,\hom](h,h')$ in this case.

\subsection{$\V$-modules}\label{subsect:Vmodules}

Besides $\V$-functors, there is another important type of morphisms between $\V$-categories, namely \emph{$\V$-modules} (also called distributors or profunctors). For $\V$-categories $X=(X,a)$ and $Y=(Y,b)$, a $\V$-module $\varphi:X\modto Y$ is a map $\varphi:X\times Y\to\V$ with a ``left $a$-action'' and a ``right $b$-action'' in the sense that, for all $x,x'\in X$ and $y,y'\in Y$,
\begin{align*}
 a(x,x')\otimes\varphi(x',y') &\le\varphi(x,y'), &
 \varphi(x,y)\otimes b(y,y') &\le \varphi(x,y').
\end{align*}
Given also $\psi:(Y,b)\modto(Z,c)$, we can calculate its composite $\psi\cdot\varphi$ with $\varphi$ as
\[
 \psi\cdot\varphi(x,z)=\bigvee_{y\in Y}\varphi(x,y)\otimes\psi(y,z),
\]
and $\psi\cdot\varphi$ is indeed a $\V$-module $\psi\cdot\varphi:(X,a)\modto(Z,c)$. By definition of a $\V$-category $(X,a)$, $a:(X,a)\modto(X,a)$ is a $\V$-module and one has $\varphi\cdot a =\varphi=b\cdot\varphi$. Hence, in the category $\Mod{\V}$ of $\V$-categories and $\V$-modules, $a:(X,a)\modto(X,a)$ is the identity $\V$-module on $X=(X,a)$ with respect to composition of $\V$-modules. The set $\Mod{\V}(X,Y)$ of $\V$-modules from $X$ to $Y$ is actually a complete ordered set where the supremum of a family $\varphi_i:X\modto Y$ ($i\in I$) is calculated point-wise. Since composition from either side preserves suprema, both monotone maps $-\cdot\varphi$ and $\varphi\cdot-$ (where $\varphi:X\modto Y$) have a right adjoint in $\ORD$. A right adjoint $-\blackleft\varphi$ of $-\cdot\varphi$ must give, for each $\psi:X\modto Z$, the largest module of type $Y\modto Z$ whose composite with $\varphi$ is contained in $\psi$,
\[
\xymatrix{X\ar[r]|-{\object@{o}}^\psi\ar[d]|-{\object@{o}}_\varphi & Z\\ Y\ar@{..>}[ur]|-{\object@{o}}^\subseteq}
\]
and we call $\psi\blackleft\varphi$ the \emph{extension of $\psi$ along $\varphi$}. Explicitly,
\[
 \psi\blackleft\varphi(y,z)=\bigwedge_{x\in X}\hom(\varphi(x,y),\psi(x,z)).
\]
Similarly, a right adjoint $\varphi\blackright-$ of $\varphi\cdot-$ must give, for each $\psi:Z\modto Y$, the largest module of type $Z\modto X$ whose composite with $\varphi$ is contained in $\psi$.
\[
\xymatrix{Y & Z\ar[l]|-{\object@{o}}_\psi\ar@{..>}[dl]|-{\object@{o}}_\supseteq\\ X\ar[u]|-{\object@{o}}^\varphi}
\]
The $\V$-module $\varphi\blackright\psi$ is called the \emph{lifting of $\psi$ along $\varphi$}, and can be calculated as
\[
\varphi\blackright\psi(z,x)=\bigwedge_{y\in Y}\hom(\varphi(x,y),\psi(z,y)).
\]
Every $\V$-functor $f:(X,a)\to(Y,b)$ gives rise to a $\V$-module $f_*:(X,a)\modto(Y,b)$ defined by
\[
 f_*(x,y)=b(f(x),y),
\]
for all $x\in X$ and $y\in Y$. Then, $(1_X)_*=a$ and, with $g:(Y,b)\to(Z,c)$, $(gf)_*=g_*\cdot f_*$, which tells us that $(-)_*$ is actually a functor
\[                                                                                                                                                  
(-)_*:\Cat{\V}\to\Mod{\V}                                                                                                                                                   \]
which leaves objects unchanged. The module $f_*:(X,a)\modto(Y,b)$ has a right adjoint $f^*:(Y,b)\modto(X,a)$ in the ordered category $\Mod{\V}$, meaning that $a\le f^*\cdot f_*$ and $f_*\cdot f^*\le b$, and $f^*$ is defined by
\[
 f^*(y,x)=b(y,f(x)).
\]
Since $(1_X)^*=a$ and $(gf)^*=f^*\cdot g^*$, this construction defines a functor
\[                                                                                                                                                  
(-)^*:\Cat{\V}^\op\to\Mod{\V}.                                                                                                                                           \]
For later use we record the calculation rules
\begin{align*}
 f^*\cdot\varphi(z,x) &=\varphi(z,f(x)), &
 \psi\cdot f_*(x,z)=\psi(f(x),z);
\end{align*}
for $\V$-modules $\varphi:Z\modto Y$ and $\psi:Y\modto Z$. It follows that $f$ is fully faithful if and only if $a=f^*\cdot f_*$, and we call $f$ \emph{fully dense} if $f_*\cdot f^*=b$.

In general, $\V$-modules $\varphi:X\modto Y$ and $\psi:Y\modto X$ form an adjunction $\varphi\dashv\psi$ if $a\le\psi\cdot\varphi$ and $\varphi\cdot\psi\le b$. As before, adjoints determine each other meaning that $\varphi\dashv\psi$ and $\varphi\dashv\psi'$ imply $\psi=\psi'$ as well as $\varphi\dashv\psi$ and $\varphi'\dashv\psi$ imply $\varphi=\varphi'$. Furthermore, one says that $\varphi$ is left adjoint whenever $\varphi\dashv\psi$ for some (unique) $\psi$, and that $\psi$ is right adjoint if $\varphi\dashv\psi$ for some (unique) $\varphi$. The following result will be extremely useful for calculating with adjoints.

\begin{lemma}\label{lem:EqAdjMod}
Let $\varphi,\varphi':X\modto Y$ and $\psi,\psi':Y\modto X$ be $\V$-modules with $\varphi\dashv\psi$, $\varphi'\dashv\psi'$, $\varphi\le\varphi'$ and $\psi\le\psi'$. Then $\varphi=\varphi'$ and $\psi=\psi'$.
\end{lemma}
\begin{proof}
By unicity of adjoints, it is enough to show $\varphi=\varphi'$, that is, $\varphi'\le\varphi$. To this end, we calculate $\varphi\ge\varphi'\cdot\psi'\cdot\varphi\ge\varphi'\cdot\psi\cdot\varphi\ge\varphi'$.
\end{proof}

Another connection between $\V$-modules and $\V$-functors is given by the fact that a map $\varphi:X\times Y\to\V$ is a $\V$-module precisely when $\varphi$ is a $\V$-functor of type $X^\op\otimes Y\to\V$. By passing to its exponential mate, we can view a $\V$-module $\varphi:(X,a)\modto(Y,b)$ as a $\V$-functor $\mate{\varphi}:Y\to[X^\op,\V]$. In particular, the $\V$-module $a:(X,a)\modto(X,a)$ corresponds to the Yoneda embedding $\yoneda_X:=\mate{a}:X\to[X^\op,\V]$, which is indeed a fully faithful $\V$-functor thanks to the \emph{Yoneda Lemma} which states that
\[
 [\yoneda_X(x),\psi]=\psi(x)
\]
for all $x\in X$ and $\psi\in[X^\op,\V]$. We also note that the set $[X^\op,\V]$ can be identified with $\Mod{\V}(X,1)$, and under this identification one has $[\psi,\psi']=\psi'\blackleft\psi$ and $\yoneda_X(x)=x^*$ (here we think of $x\in X$ as $x:1\to X$). If $\psi$ has a left adjoint $\varphi$, then $-\cdot\psi\dashv -\cdot\varphi$ and therefore $[\psi,\psi']=\psi'\cdot\varphi$, for all $\psi':X\modto 1$. In particular, for $\psi'=x^*$ we obtain $[\psi,x^*]=\varphi(x)$.

Another way to read the Yoneda Lemma goes as it follows: for any $\V$-module $\psi:X\modto 1$, seen also as an element of $[X^\op,\V]$, one has $\psi^*\cdot(\yoneda_X)_*=\psi$. If, moreover, $\psi$ has a left adjoint $\varphi$, then also $\yoneda_X^*\cdot\psi_*=[\psi,(-)^*]=\varphi$, and therefore $\psi^*\ge\psi\cdot\yoneda_X^*$ and $\psi_*\ge(\yoneda_X)_*\cdot\varphi$. Hence, Lemma \ref{lem:EqAdjMod} implies
\begin{lemma}\label{lem:PsiRepresentPsi}
For every adjunction $(\varphi:1\modto X)\dashv(\psi:X\modto 1)$, $\psi^*=\psi\cdot\yoneda_X^*$ and $\psi_*=(\yoneda_X)_*\cdot\varphi$.
\end{lemma}

Finally, we note that every morphism of quantales $F:\V\to\W$ induces also a functor $F:\Mod{\V}\to\Mod{\W}$ which extends $F:\Cat{\V}\to\Cat{\W}$ in the sense that both diagrams
\begin{align}\label{eq:CommDiagF*}
\xymatrix{\Mod{\V}\ar[r]^F & \Mod{\W}\\
 \Cat{\V}\ar[u]^{(-)_*}\ar[r]_F & \Cat{\W}\ar[u]_{(-)_*}}
&&
\xymatrix{\Mod{\V}\ar[r]^F & \Mod{\W}\\
 \Cat{\V}^\op\ar[u]^{(-)^*}\ar[r]_{F^\op} & \Cat{\W}^\op\ar[u]_{(-)^*}}
\end{align}
commute. Here, for a $\V$-module $\varphi:(X,a)\modto(Y,b)$, $F\varphi$ is defined as the $\W$-module of type $F(X,a)\modto F(Y,b)$ given by the composite
\[
 X\times Y\xrightarrow{\;\varphi\;}\V\xrightarrow{\;F\;}\W.
\]
Furthermore, $F$ is even locally monotone meaning here that $\varphi\le\varphi'$ implies $F\varphi\le F\varphi'$, and therefore one has $F\varphi\dashv F\psi$ in $\Mod{\W}$ for every adjunction $\varphi\dashv\psi$ in $\Mod{\V}$.

\section{Probabilistic metric spaces}\label{sect:ProbMetSp}

\subsection{The quantale $\Delta$ of distribution functions}\label{subsect:Delta}

In this paper we are particularly interested in
\[
\Delta=\{f:[0,\infty]\to [0,1]\mid f \text{ is monotone and } f(x)=\bigvee_{y<x} f(y) \},
\]
which is a complete lattice with the point-wise order. To see that $\Delta$ is completely distributive, consider the maps $f_{\delta,u} \in \Delta$, with $0\leq u \leq 1$ and $0\leq \delta$, defined by:
\[
f_{\delta,u}(t)=
\begin{cases}
0 & \text{ if } 0\leq t \leq \delta,\\
u & \text{ if } \delta < t.
\end{cases}
\]
Then for any $f\in\Delta$ one has
\[
 f_{\delta,u}\ll f \iff u<f(\delta),
\]
and $f=\bigvee\{f_{\delta,u}\mid f_{\delta,u}\ll f\}$.

The complete lattice $\Delta$ becomes a quantale where, for $f,g\in\Delta$ and $t\in[0,\infty]$,
\[
 f\otimes g(t)=\bigvee_{r+s\le t}f(r)\cdot g(s).
\]
One easily verifies that $\otimes$ is associative and commutative, and that
\[
 \eps:=f_{0,1}:[0,\infty]\to [0,1],\, t\mapsto
\begin{cases}
 0 & \text{if }t=0,\\
 1 & \text{else}
\end{cases}
\]
is a neutral element for $\otimes$. Furthermore, $f_{\delta,u}\otimes f_{\delta',u'}=f_{\delta+\delta',u\cdot u'}$.

For any $f\in\Delta$, $f\otimes-:\Delta\to\Delta$ preserves suprema since $u\cdot-:[0,1]\to[0,1]$ does so, for all $u\in[0,1]$. Therefore $f\otimes-$ has a right adjoint $\hom(f,-):\Delta\to\Delta$, where (with $g\in\Delta$)
\[
\hom(f,g)=\bigvee\{h \in \Delta \mid f \otimes h \leq g\}
=\bigvee\{h \in \Delta\mid \forall r,s,t\in[0,\infty]\,.\,(r+s\le t\,\Rw\, f(r) \cdot h(s) \leq g(t))\}.
\]
If $f=f_{\delta,u}$, then $\hom(f_{\delta,u},g)$ is given by the supremum over all those $h\in\Delta$ which satisfy, for all $t\in[0,\infty]$, $u\cdot h(t\trunminus\delta)\le g(t)$, which by adjunction is equivalent to $h(t\trunminus\delta)\le g(t)\trundiv u$. If, moreover, $g=f_{\delta',u'}$, then this supremum is actually obtained for $h=f_{\delta'\trunminus\delta,u'\trundiv u}$, that is, $\hom(f_{\delta,u},f_{\delta',u'})=f_{\delta'\trunminus\delta,u'\trundiv u}$.

We call $f\in\Delta$ \emph{finite} if $f(\infty)=1$. Certainly, if $f$ is finite, then so is every $g\in\Delta$ with $f\le g$; and one also easily verifies that $f\otimes g$ is finite if both $f,g\in\Delta$ are so.

\subsection{Probabilistic metric spaces}

A probabilistic metric space \cite{Men42,SS83} is classically defined relative to a so-called \emph{t-norm}, which is nothing but a quantale structure $*:[0,1]\times[0,1]\to[0,1]$ on $[0,1]$ with neutral element $1$. Given a t-norm $*$, a \emph{probabilistic metric space} is a set $X$ equipped with a distance function $d:X\times X\times[0,\infty]\to[0,1]$ subject to
\begin{enumerate}
\item\label{cond:ProbMet1} $d(x,y,-):[0,\infty]\to[0,1]$ is left continuous (that is, $d(x,y,t)=\bigvee_{s<t}d(x,y,s)$),
\item\label{cond:ProbMet2} $d(x,x,t)=1$ for $t>0$,
\item\label{cond:ProbMet3} $d(x,y,r)*d(y,z,s)\le d(x,z,r+s)$,
\item\label{cond:ProbMet4} $d(x,y,t)=1=d(y,x,t)$ for all $t>0$ implies $x=y$,
\item\label{cond:ProbMet5} $d(x,y,t)=d(y,x,t)$ for all $t$,
\item\label{cond:ProbMet6} $d(x,y,\infty)=1$,
\end{enumerate}
for all $x,y\in X$ and $r,s\in[0,\infty]$. The intended meaning of $d(x,y,t)=u$ is that $u$ is the ``probability of the distance from $x$ to $y$ is less then $t$''. For the sake of simplicity, in the sequel we will only consider the case of $*$ being the usual multiplication ``$\cdot$'' on $[0,1]$. Clearly, \eqref{cond:ProbMet1} just states that the exponential mate $a:=\mate{d}:X\times X\to[0,1]^{[0,\infty]}$ of $d$ takes values in $\Delta$. Then \eqref{cond:ProbMet2} and \eqref{cond:ProbMet3} are indeed equivalent to the defining properties
\begin{align*}
 \eps&\le a(x,y), & a(x,y)\otimes a(y,z)&\le a(x,z)
\end{align*}
of a $\Delta$-category. In the sequel we will use the term ``probabilistic metric space'' as a synonym for $\Delta$-category, hence we do not insist on the conditions \eqref{cond:ProbMet4}, \eqref{cond:ProbMet5} and \eqref{cond:ProbMet6}. However, we note that a $\Delta$-category $X=(X,a)$ satisfies \eqref{cond:ProbMet4} if and only if $X$ is separated, and $X$ satisfies \eqref{cond:ProbMet5} if and only if $X$ is symmetric. Similarly to the nomenclature for metric spaces, we call a $\Delta$-category $X=(X,a)$  \emph{finitary} if $X$ satisfies \eqref{cond:ProbMet6}, i.e.\ if $a(x,y)\in\Delta$ is finite for all $x,y\in X$. Intuitively, \eqref{cond:ProbMet6} states that the affirmation ``the distance from $x$ to $y$ is finite'' has probability $1$.

Finally, a $\Delta$-functor $f:(X,a)\to(Y,b)$ is a map satisfying
\begin{equation}\label{cond:DeltaFun}
 a(x,y)(t)\le b(f(x),f(y))(t)
\end{equation}
for all $x,y\in X$ and $t\in[0,\infty]$. In other words, the ``probability of the distance from $x$ to $y$ is less than $t$'' is less or equal than the ``probability of the distance from $f(x)$ to $f(y)$ is less than $t$''. We write $\PROBMET$ for the category of probabilistic metric spaces and maps satisfying \eqref{cond:DeltaFun}, that is, $\PROBMET\simeq\Cat{\Delta}$.

\begin{remark}
The notion of probabilistic metric spaces is closely related to the one of fuzzy metric space as defined in \cite{GV94}. The main difference appears in condition \eqref{cond:ProbMet1}: in \cite{GV94}, the mapping $d(x,y,-)$ is even required to be continuous. However, with this modification the set of all distribution functions is not any more complete in the pointwise order. As a consequence, many nice properties of probabilistic metric spaces are not shared by fuzzy metric spaces, for instance, there exist fuzzy metric spaces which do not admit a Cauchy completion (see \cite{GR02}).
\end{remark}

\subsection{Comparison with metric spaces}

The quantale $[0,\infty]$ embeds canonically into $\Delta$ via
\[
 I_\infty:[0,\infty]\to\Delta,\,x\mapsto f_{x,1}.
\]
Moreover, for all $x,y\in[0,\infty]$,
\begin{align*}
 f_{0,1}&=\eps, & f_{x+y,1}&=f_{x,1}\otimes f_{y,1},
\end{align*}
and $I_\infty$ preserves suprema since it has a right adjoint
\[
 P_\infty:\Delta\to[0,\infty],\,f\mapsto\inf\{x\in[0,\infty]\mid f(x)=1\}.
\]
Consequently, $I_\infty$ is a morphism of quantales. The right adjoint $P_\infty$ satisfies
\begin{align*}
 P_\infty(\eps)&=0, & P_\infty(f\otimes g)=P_\infty(f)+P_\infty(g),
\end{align*}
for all $f,g\in\Delta$; however, $P_\infty$ does not preserve suprema and therefore is only a lax morphism of quantales. Furthermore, $I_\infty$ has also a left adjoint
\[
 O_\infty:\Delta\to[0,\infty],\,f\mapsto\sup\{x\in[0,\infty]\mid f(x)=0\},
\]
which -- being left adjoint -- preserves suprema and also satisfies
\begin{align*}
 O_\infty(\eps)&=0, & O_\infty(f\otimes g)=O_\infty(f)+O_\infty(g),
\end{align*}
for all $f,g\in\Delta$. Therefore $O_\infty$ is a morphism of quantales. Finally, from
\[
 \xymatrix{[0,\infty]\ar[rr]|-{I_\infty} &&
\Delta\ar@/^3.5ex/[ll]_\perp^-{P_\infty}\ar@/_3.5ex/[ll]^\perp_-{O_\infty}}
\]
one obtains the chain of functors

\[
 \xymatrix{\MET\ar[rr]|-{I_\infty} &&
\PROBMET.\ar@/^3.5ex/[ll]_\perp^-{P_\infty}\ar@/_3.5ex/[ll]^\perp_-{O_\infty}}
\]

\section{Cauchy completeness}

\subsection{Cauchy complete $\V$-categories}

In Subsection \ref{subsect:Vmodules} we mentioned already that each $\V$-functor $f:(X,a)\to(Y,b)$ induces an adjoint pair $f_*\dashv f^*$ of $\V$-modules. One of the amazing discoveries of \cite{Law_MetLogClo} is that the reverse affirmation (every adjoint pair of $\V$-modules is induced by a $\V$-functor) is ultimately linked to Cauchy completeness. In fact, in \cite{Law_MetLogClo} it is shown that
\begin{theorem}
A metric space $X$ (viewed as a $[0,\infty]$-category) is Cauchy complete if and only if every adjunction $\varphi\dashv\psi$, $\varphi:Y\modto X$, $\psi:X\modto Y$ in $\Mod{[0,\infty]}$ is of the form $f_*\dashv f^*$, for a non-expansive map $f:Y\to X$.
\end{theorem}
It is not hard to see that it is enough to consider $Y=1$ the one-point space in the theorem above. By definition $\varphi:1\modto X$ is left adjoint to $\psi:X\modto 1$ precisely when
\begin{align*}
 0&\geqslant\inf_{x\in X}\varphi(x)+\psi(x), & \psi(x)+\varphi(y)\geqslant d(x,y)
\end{align*}
for all $x,y\in X$, where $d$ denotes the metric of $X$. The main observation here is that there is a bijection between adjunctions $\varphi\dashv\psi$ and equivalence classes of Cauchy sequences, realised by
\begin{align*}
 (x_n)_{_\N}\quad\text{Cauchy sequence in $X$} &\mapsto
\begin{cases}
\varphi:1\modto X, &\varphi(x)=\displaystyle{\sup_{n\in\N}\inf_{k\geqslant n}d(x_n,x)}\\
\psi:X\modto 1, &\psi(x)=\displaystyle{\sup_{n\in\N}\inf_{k\geqslant n}d(x,x_n)}\\
\varphi\dashv\psi
\end{cases}\\
\varphi\dashv\psi &\mapsto
(x_n)_{_\N}\quad\text{with $x_n$ such that $\varphi(x_n)+\psi(x_n)\leqslant\frac{1}{n}$}.
\end{align*}
Furthermore, a Cauchy sequence $(x_n)_{_\N}$ converges to $x$ if and only if the corresponding adjunction $\varphi\dashv\psi$ is of the form $x_*\dashv x^*$.

In general, one says that a $\V$-category $X$ is \emph{Cauchy complete} whenever every adjunction
\[
 (\varphi:Y\modto X)\dashv(\psi:X\modto Y)
\]
is representable by some $\V$-functor $f:Y\to X$ in the sense that $\varphi=f_*$ and $\psi=f^*$. As above, it is enough to consider the case $Y=E=(1,k)$. Since adjoints determine each other, $X$ is Cauchy complete whenever every left adjoint $\V$-module $\varphi:E\modto X$ is of the form $\varphi=x_*$, equivalently, whenever every right adjoint $\V$-module $\psi:X\modto E$ is of the form $\psi=x^*$.

For any $\V$-category $X$, $[X^\op,\V]$ is Cauchy complete where a right adjoint $\V$-module $\psi:[X^\op,\V]\modto E$ is represented by $\psi\cdot(\yoneda_X)_*\in[X^\op,\V]$. Furthermore, when writing
\[
 \widetilde{X}=\{\psi:X\modto E\mid\text{$\psi$ is right adjoint}\}
\]
for the $\V$-subcategory of $[X^\op,\V]$ defined by all right adjoint $\V$-modules, $X$ is Cauchy complete if and only if the restriction $\yoneda_X:X\to\widetilde{X}$ of the Yoneda embedding to $\widetilde{X}$ is surjective. In fact, $\yoneda_X:X\to\widetilde{X}$ provides a Cauchy completion of the $\V$-category $X$ as we recall in the next Subsection.

Finally, we study briefly how functors induced by morphisms of quantales interact with Cauchy completeness. To this end, let $F:\V\to\W$ be a morphism of quantales $\V=(\V,\otimes,k)$ and $\W=(\W,\oplus,l)$ and let $X$ be a $\V$-category. Recall that $F$ induces functors $F:\Cat{\V}\to\Cat{\W}$ and $F:\Mod{\V}\to\Mod{\W}$, and the latter one sends adjunctions to adjunctions. Therefore one obtains a commutative diagram
\[
 \xymatrix{\{\varphi:E\modto X\mid \varphi\text{ is left adjoint}\}\ar[rr]^\Phi &&
 \{\varphi':E=FE\modto FX\mid \varphi'\text{ is left adjoint}\}\\
 & |X|=|FX|,\ar[ul]^{(-)_*}\ar[ur]_{(-)_*}}
\]
where $\Phi\varphi=F\varphi$ and $|Y|$ denotes the underlying set of a $\V$-category $Y$. Since $X$ (respectively $FX$) is Cauchy complete if and only if the map $(-)_*$ is surjective, we find
\begin{proposition}\label{prop:PreservCauchyViaPhi}
\begin{enumerate}
\item If $FX$ is Cauchy complete and $\Phi$ is injective, then $X$ is Cauchy complete.
\item If $X$ is Cauchy complete and $\Phi$ is surjective, then $FX$ is Cauchy complete.
\end{enumerate}
\end{proposition}

Certainly, if $F:\V\to\W$ is injective, then $\Phi$ is injective for every $\V$-category $X$. In order to obtain surjectivity of $\Phi$, we also assume that there is a morphism of quantales $G:\W\to\V$. Hence, $(\varphi':E\modto FX)\dashv(\psi':FX\modto E)$ in $\Mod{\W}$ gives $(G\varphi':E\modto GFX)\dashv(G\psi':GFX\modto E)$ in $\Mod{\V}$, and $G\varphi'$ is of type $E\modto X$ provided that $GF=1_\V$. Moreover, if either $FG\le 1_\W$ or $FG\ge 1_\W$, then $FG\varphi'\le\varphi'$ and $FG\psi'\le\psi'$ or $FG\varphi'\ge\varphi'$ and $FG\psi'\ge\psi'$, and either way Lemma \ref{lem:EqAdjMod} implies $FG\varphi'=\varphi'$. A similar argument can be used if $GF\le 1_\V$ and $FG\ge 1_\W$ (that is, $G\dashv F$), since in this case the identity map on $X$ is a $\V$-functor of type $\gamma:GFX\to X$, and $(G\varphi':E\modto GFX)\dashv(G\psi':GFX\modto E)$ can be composed with $\gamma_*\dashv\gamma^*$ to yield $(\gamma_*\cdot G\varphi':E\modto X)\dashv(G\psi'\cdot\gamma^*:X\modto E)$. Furthermore, $F\gamma$ is the identity on $FX$ since $FGF=F$, and $\Phi(\gamma_*\cdot G\varphi')=\varphi'$ follows again from Lemma \ref{lem:EqAdjMod}.

\begin{corollary}\label{cor:XtoFXCauchy}
Let $F:\V\to\W$ and $G:\W\to\V$ be morphisms of quantales and assume that either $G\dashv F$ or that $F\dashv G$ and $F$ is injective. Then $FX$ is Cauchy complete provided that $X$ is Cauchy complete.
\end{corollary}

\subsection{Topology in a $\V$-category}

To every metric on a set $X$ one associates a topology by putting
\[
 x\in\overline{M}:\iff\text{there is some (Cauchy) sequence $(x_n)_{_\N}$ in $M$ with $(x_n)_{_\N}\to x$,}
\]
for all $M\subseteq X$ and $x\in X$. Rephrased in the language of $\V$-modules, $x$ is in the closure of $M$ if and only if $x$ represents an adjoint pair of $\V$-modules on $M$, and this amounts to saying that
\[
 (i^*\cdot x_*:E\modto M)\dashv(x^*\cdot i_*):M\modto E,
\]
where we consider $M$ as a sub-$\V$-category of $X$ and $i:M\to X$ denotes the inclusion $\V$-functor. This latter formulation defines indeed a closure operator (not just for a metric space but) for any $\V$-category $X$ which was studied in \cite{HT_LCls}. Below we recall some key facts; if not stated otherwise, their proofs can be found in \cite{HT_LCls}.
\begin{proposition}
Let $X=(X,a)$ be a $\V$-category, $M\subseteq X$ and $x\in X$. Then
\[
 x\in\overline{M}\iff k\le\bigvee_{y\in M}a(x,y)\otimes a(y,x).
\]
\end{proposition}
By the proposition above, for $x,x'\in\overline{M}$ one has
\[
 a(x,x')\le \bigvee_{y\in M}a(x,y)\otimes a(y,x)\otimes a(x,x')
\le \bigvee_{y\in M}a(x,y)\otimes a(y,x')\le a(x,x'),
\]
hence $a(x,x')=\bigvee_{y\in M}a(x,y)\otimes a(y,x')$.

\begin{proposition}\label{prop:PropertiesClosure}
Let $f:X\to Y$ be a $\V$-functor, $M,M'\subseteq X$ and $N\subseteq
Y$. Then
\begin{enumerate}
\item $M\subseteq\overline{M}$,
\item $M\subseteq M'$ implies $\overline{M}\subseteq\overline{M'}$,
\item $\overline{\varnothing}=\varnothing$ and
$\overline{\overline{M}}=\overline{M}$,
\item $f(\overline{M})\subseteq \overline{f(M)}$ and
$\overline{f^{-1}(N)}\subseteq f^{-1}(\overline{N})$,
\item $\overline{M\cup M'}=\overline{M}\cup\overline{M'}$ provided that $k\leq u\vee v$ implies $k\leq u$ or $k\leq v$ for all $u,v\in\V$.
\end{enumerate}
Furthermore, $\overline{(-)}$ is hereditary, that is, for $M\subseteq Z\subseteq X$ where we consider $Z$ as a sub-$\V$-category of $X$, $\overline{M}$ calculated in the $\V$-category $Z$ is equal to $\overline{M}\cap Z$ with $\overline{M}$ calculated in the $\V$-category $X$.
\end{proposition}

Although this closure operator is in general not topological, it still allows us to introduce a notion of convergence. For an ultrafilter $\mathfrak{x}$ on $X$ and a point $x\in X$, one says that $\mathfrak{x}$ converges to $x$, written as $\mathfrak{x}\rw x$, whenever $x\in\overline{A}$ for all $A\in\mathfrak{x}$. Then a filter $\mathfrak{f}$ converges to $x$, $\mathfrak{f}\rw x$, if every ultrafilter $\mathfrak{x}$ with $\mathfrak{f}\subseteq\mathfrak{x}$ converges to $x$. In particular, for a sequence $s=(x_n)_{n\in\N}$ and a point $x\in X$, we find that
\[
 s\rw x\iff x\in\overline{\{x_n\mid n\in M\}},\text{ for all $M\subseteq\N$ infinite}.
\]

One also has the expected results linking closed subsets $M\subseteq X$ (i.e.\ $\overline{M}=M$) with Cauchy completeness: every closed subset of a Cauchy complete $\V$-category is Cauchy complete, and every Cauchy complete $\V$-subcategory of a separated $\V$-category is closed. We also note that the inclusion $\V$-functor $i:M\to X$ is fully dense (i.e.\ $i_*\cdot i^*=a$ where $X=(X,a)$) if and only if $\overline{M}=X$.

An important example is provided by the Yoneda embedding $\yoneda_X:X\to[X^\op,\V]$, since 
\[
 \overline{\yoneda_X(X)}=\widetilde{X}=\{\psi:X\modto 1\mid \psi\text{ is right adjoint}\}.
\]
Hence, $\widetilde{X}$ is Cauchy complete and $\yoneda_X:X\to\tilde{X}$ is (fully faithful and) fully dense. This makes $(\yoneda_X)_*:X\modto\tilde{X}$ an isomorphism in $\Mod{\V}$ with inverse $\yoneda_X^*:\tilde{X}\modto X$, and from that it follows at once that $\yoneda_X:X\to\tilde{X}$ has the desired universal property: for every $\V$-functor $f:X\to Y$ where $Y$ is separated and Cauchy complete, there exists a unique $\V$-functor $g:\tilde{X}\to Y$ with $g\cdot\yoneda_X=f$. In fact, $g$ can be taken as the $\V$-functor $\tilde{X}\to Y$ which represents the left adjoint $\V$-module $f_*\cdot\yoneda_X^*$, such a $\V$-functor exists since $Y$ is Cauchy complete and is unique since $Y$ is separated.

\begin{theorem}\label{thm:CauchyReflective}
The full subcategory of separated and Cauchy complete $\V$-categories is reflective in $\Cat{\V}$, where the reflection map for a $\V$-category $X$ can be chosen as $\yoneda_X:X\to\tilde{X}$.
\end{theorem}

The discussion preceding Theorem \ref{thm:CauchyReflective} applies actually to any fully faithful and fully dense $\V$-functor in lieu of the Yoneda embedding, which gives
\begin{proposition}
A $\V$-category $X$ is Cauchy complete if and only if $X$ is injective with respect to fully faithful and fully dense $\V$-functors. 
\end{proposition}
Here a $\V$-category $X$ is injective with respect to fully faithful and fully dense $\V$-functors whenever, for every fully faithful and fully dense $i:A\to B$ and every $\V$-functor $f:A\to X$, there exists a $\V$-functor $g:B\to X$ with $g\cdot i\simeq f$. The proposition above provides us with an alternative way to prove preservation of Cauchy completeness by functors.
\begin{corollary}\label{cor:PreservCauchyInj}
Let $G:\Cat{\V}\to\Cat{\W}$ be a locally monotone functor with left adjoint $F:\Cat{\W}\to\Cat{\V}$. If $F$ sends fully faithful and fully dense $\W$-functors to fully faithful and fully dense $\V$-functors, then $G$ sends Cauchy complete $\V$-categories to Cauchy complete $\W$-categories.
\end{corollary}
\begin{proof}
We write $\eta:1\to FG$ and $\eps:FG\to 1$ for the units of the adjunction $F\dashv G$. Let $X$ be a Cauchy complete $\V$-category, $i:A\to B$ be a fully dense embedding in $\Cat{\W}$ and $f:A\to GX$ be a $\V$-functor. By hypothesis, $Fi$ is a fully dense embedding, hence there exists some $\V$-functor $g:FB\to X$ with $g\cdot Fi\simeq\eps_X\cdot Ff$. Then
\[
 Gg\cdot\eta_B\cdot i=Gg\cdot GFi\cdot\eta_A
\simeq G\eps_X\cdot GFf\cdot\eta_A
=G\eps\cdot\eta_{GX}\cdot f=f.\qedhere
\]
\end{proof}

\begin{lemma}\label{lem:PreservFullyDense}
Let $F:\W\to\V$ be a morphism of quantales. Then $F:\Cat{\W}\to\Cat{\V}$ sends fully faithful and fully dense $\W$-functors to fully faithful and fully dense $\V$-functors.
\end{lemma}
\begin{proof}
This follows immediately from the commutativity of the diagrams \eqref{eq:CommDiagF*} in Subsection \ref{subsect:Vmodules}.
\end{proof}

\subsection{Cauchy sequences in a $\V$-category}

For a sequence $s=(x_n)_{n\in\N}$ in a $\V$-category $X=(X,a)$, one defines (see \cite{Wag_PhD})
\[
 \Cauchy_X(s)=\bigvee_{N\in\N}\bigwedge_{n,m\geq N}a(x_n,x_m),
\]
which should be seen as a measure of ``Cauchyness'' of $s$. Note that $\Cauchy_X(s)=\Cauchy_{X^\op}(s)$, and $\Cauchy_X(s)=\Cauchy_Y(s)$ for every $\V$-category $Y$ having $X$ as a sub-$\V$-category. More generally, for a $\V$-functor $f:X\to Y$ and a sequence $(x_n)_{n\in\N}$, one has $\Cauchy_X(s)\le\Cauchy_Y(f(s))$ where $f(s)$ denotes the sequence $(f(x_n))_{n\in\N}$ in $Y$, and this inequality is even an equality if $f$ is fully faithful. 

In the sequel we will simply write $\Cauchy(s)$ if it is understood from the context which $\V$-category we consider. Furthermore, one says that $s$ is \emph{Cauchy} in $X$ if $k\le\Cauchy(s)$. By the considerations above, every $\V$-functor sends Cauchy sequences to Cauchy sequences.

\begin{lemma}\label{lem:Subseq_is_Cauchy}
If $s'$ is a subsequence of $s$ in a $\V$-category $X$, then $\Cauchy(s')\ge\Cauchy(s)$. In particular, every subsequence of a Cauchy sequence is Cauchy.
\end{lemma}
\begin{proof}
Let $s=(x_n)_{n\in\N}$ be a sequence in $X=(X,a)$, $M\subseteq\N$ be an infinite subset and $s'=(x_m)_{m\in M}$. Then
\[
 \Cauchy(s')
=\bigvee_{N\in M}\bigwedge_{\substack{n,m\ge N,\\ n,m\in M}}a(x_n,x_m)
\ge\bigvee_{N\in M}\bigwedge_{n,m\ge N}a(x_n,x_m)
=\bigvee_{N\in\N}\bigwedge_{n,m\ge N}a(x_n,x_m).\qedhere
\]
\end{proof}

\begin{lemma}
For any sequence $s=(x_n)_{n\in\N}$ in a $\V$-category $X=(X,a)$ and any $x\in X$:
\[
 \bigvee_{N\in\N}\bigwedge_{n\geq N}a(x_n,x)\geq\Cauchy(s)\otimes \bigwedge_{N\in\N}\bigvee_{n\geq N}a(x_n,x)
\]
\end{lemma}
\begin{proof}
We calculate:
\begin{align*}
\Cauchy(s)\otimes \bigwedge_{N\in\N}\bigvee_{n\geq N}a(x_n,x)
&=\left(\bigvee_{N\in\N}\bigwedge_{n,m\geq N}a(x_n,x_m)\right)\otimes \left(\bigwedge_{N'\in\N}\bigvee_{k\geq N'}a(x_k,x)\right)\\
&=\bigvee_{N\in\N}\left(\bigwedge_{n,m\geq N}a(x_n,x_m)\otimes \bigwedge_{N'\in\N}\bigvee_{k\geq N'}a(x_k,x)\right)\\
&\le\bigvee_{N\in\N}\left(\bigwedge_{n,m\geq N}a(x_n,x_m)\otimes \bigvee_{k\geq N}a(x_k,x)\right)\\
&=\bigvee_{N\in\N}\bigvee_{k\geq N}\left(\bigwedge_{n,m\geq N}(a(x_n,x_m)\otimes a(x_k,x))\right)\\
&\le\bigvee_{N\in\N}\bigvee_{k\geq N}\left(\bigwedge_{n\geq N}(a(x_n,x_k)\otimes a(x_k,x))\right)\\
&\le\bigvee_{N\in\N}\bigwedge_{n\geq N}a(x_n,x).\qedhere
\end{align*}
\end{proof}

\begin{corollary}
For a Cauchy sequence $(x_{n})_{n\in\N}$ in a $\V$-category $X=(X,a)$ and $x \in X$,
\begin{align*}
\bigvee_{N\in\N}\bigwedge_{n\geq N}a(x_n,x)&=\bigwedge_{N\in\N}\bigvee_{n\geq N}a(x_n,x)
&\text{and}&&
\bigvee_{N\in\N}\bigwedge_{n\geq N}a(x,x_n)&=\bigwedge_{N\in\N}\bigvee_{n\geq N}a(x,x_n).
\end{align*}
\end{corollary}

To any sequence $s=(x_n)_{n\in\N}$ in a $\V$-category $X=(X,a)$ one can associate $\V$-modules $\varphi_s:1\modto X$ and $\psi_s:X\modto 1$ defined as
\begin{align*}
\varphi_s(x)&=\bigvee_{N\in\N}\bigwedge_{n\geq N}a(x_n,x)
&\text{and}&&
\psi_s(x)&=\bigvee_{N\in\N}\bigwedge_{n\geq N}a(x,x_n),
\end{align*}
for all $x\in X$. In fact, one easily verifies that $\varphi_s$ and $\psi_s$ are $\V$-modules, moreover, one has
\begin{align*}
\psi_s(x)\otimes\varphi_s(y)
&=\left(\bigvee_{N\in\N}\bigwedge_{n\geq N}a(x,x_n)\right)\otimes
\left(\bigvee_{N\in\N}\bigwedge_{n\geq N}a(x_n,y)\right)\\
&=\bigvee_{N\in\N}\left(\bigwedge_{n\geq N}a(x,x_n)\otimes\bigwedge_{m\geq N}a(x_m,y)\right)\\
&\le\bigvee_{N\in\N}a(x,x_N)\otimes a(x_N,y)\le a(x,y)
\intertext{for all $x,y\in X$, and}
\bigvee_{x\in X}\varphi_s(x)\otimes\psi_s(x)
&=\bigvee_{x\in X}\bigvee_{N\in\N}\left(\bigwedge_{m\geq N}a(x_m,x)\otimes\bigwedge_{n\geq N}a(x,x_n)\right)\\
&\ge\bigvee_{N\in\N}\left(\bigwedge_{m\geq N}a(x_m,x_N)\otimes\bigwedge_{n\geq N}a(x_N,x_n)\right)\\
&=\left(\bigvee_{N\in\N}\bigwedge_{m\geq N}a(x_m,x_N)\right)\otimes\left(\bigvee_{N\in\N}\bigwedge_{n\geq N}a(x_N,x_n)\right)\\
&\ge\Cauchy(s)\otimes\Cauchy(s).
\end{align*}
Furthermore, for every $x\in X$,
\begin{align*}
\left(\bigvee_{N\in\N}\bigwedge_{n\ge N}a(x_n,x)\right)\otimes\left(\bigvee_{M\in\N}\bigwedge_{m\ge M}a(x,x_m)\right)
&=\bigvee_{N\in\N}\left(\bigwedge_{n\ge N}a(x_n,x)\otimes\bigwedge_{m\ge N}a(x,x_m)\right)\\
&\le\bigvee_{N\in\N}\bigwedge_{n,m\ge N}a(x_n,x)\otimes a(x,x_m)\\
&\le\bigvee_{N\in\N}\bigwedge_{n,m\ge N}a(x_n,x_m),
\end{align*}
hence $\displaystyle{\bigvee_{x\in X}\varphi_s(x)\otimes\psi_s(x)\le\Cauchy(s)}$. All told,
\begin{proposition}
Let $s$ be a sequence in a $\V$-category $X$. Then $s$ is Cauchy if and only if $\varphi_s\dashv\psi_s$ in $\Mod{\V}$.
\end{proposition}

\begin{lemma}\label{lem:f_preserves_module}
Let $f:X\to Y$ be a $\V$-functor, $s=(x_n)_{n\in\N}$ be a Cauchy sequence in $X$ with associated adjunction $\varphi_s\dashv\psi_s$ in $\Mod{\V}$. Then $\varphi_{f(s)}=f_*\cdot\varphi_s$ and $\psi_{f(s)}=\psi_s\cdot f^*$, where $f(s)$ denotes the sequence $(f(x_n))_{n\in\N}$ in $Y$.
\end{lemma}
\begin{proof}
By Lemma \ref{lem:EqAdjMod}, it is enough to show that $\varphi_{f(s)}\ge f_*\cdot\varphi_s$ and $\psi_{f(s)}\ge\psi_s\cdot f^*$. In fact, for any $y\in Y$ (and with $X=(X,a)$ and $Y=(Y,b)$),
\begin{multline*}
 f_*\cdot\varphi_s(y)
=\bigvee_{x\in X}\varphi_s(x)\otimes b(f(x),y)
=\bigvee_{x\in X}\left(\bigvee_{N\in\N}\bigwedge_{n\ge N}a(x_n,x)\right)\otimes b(f(x),y)\\
\le\bigvee_{x\in X}\bigvee_{N\in\N}\bigwedge_{n\ge N}b(f(x_n),f(x))\otimes b(f(x),y)
\le\varphi_{f(s)}(y),
\end{multline*}
and the other inequality follows similarly.
\end{proof}

\begin{lemma}
Let $s$ be a Cauchy sequence in a $\V$-category $X$ and $s'$ be a subsequence of $s$. Then $\varphi_s=\varphi_{s'}$ and $\psi_s=\psi_{s'}$, where $\varphi_s,\varphi_{s'},\psi_s$ and $\psi_{s'}$ denote the associated $\V$-modules.
\end{lemma}
\begin{proof}
By Lemma \ref{lem:Subseq_is_Cauchy}, $s'$ is also Cauchy and therefore $\varphi_{s'}\dashv\psi_{s'}$. An easy calculation shows that $\varphi_s\le\varphi_{s'}$ and $\psi_s\le\psi_{s'}$, and the assertion follows from Lemma \ref{lem:EqAdjMod}.
\end{proof}

\begin{proposition}\label{prop:Repr_Implies_Conv}
Let $s$ be a Cauchy sequence in a $\V$-category $X$, and assume that the associated adjunction $\varphi_s\dashv\psi_s$ is of the form $x_*\dashv x^*$, for some $x\in X$. Then $s$ converges to $x$. 
\end{proposition}
\begin{proof}
We have to show that $x\in\overline{\{x_n\mid n\in M\}}$, for every $M\subseteq\N$ infinite. Since every subsequence of $s$ induces the same adjunction $\varphi_s\dashv\psi_s$, it is enough to consider $M=\N$. Let $A=\{x_n\mid n\in\N\}$ and $i:A\hrw X$ be the inclusion $\V$-functor. Then $i^*\cdot x_*:1\modto A$ and $x^*\cdot i_*:A\modto 1$ are the $\V$-modules induced by the Cauchy sequence $s$ in $M$, hence $i^*\cdot x_*\dashv x^*\cdot i_*$, that is, $x\in\overline{A}$.
\end{proof}

Hence, representability of the corresponding adjunction $\varphi_s\dashv\psi_s$ guarantees convergence of a Cauchy sequence $s$. We investigate now under which conditions the reverse statement is true.

\begin{proposition}
Let $X$ be a $\V$-category and $s=(x_n)_{n\in\N}$ be a Cauchy sequence in $X$. Put $\widetilde{s}:=(x_n^*)_{n\in\N}$. Then $\widetilde{s}\rw\psi_s$ in $[X^\op,\V]$. If, moreover, $k=\top$ is the top-element in $\V$, then $\widetilde{s}\rw\psi$ implies $\psi=\psi_s$, for every $\V$-module $\psi\in[X^\op,\V]$.
\end{proposition}
\begin{proof}
By Lemma \ref{lem:PsiRepresentPsi}, $\psi_s^*=\psi_s\cdot\yoneda_X^*$, and $\psi_s\cdot\yoneda_X^*=\psi_{\widetilde{s}}$ by Lemma \ref{lem:f_preserves_module}. Hence, by Proposition \ref{prop:Repr_Implies_Conv}, $\widetilde{s}\rw\psi_s$. For the second statement, assume now that $k=\top$ is the top-element in $\V$. First note that from $\widetilde{s}\rw\psi$ it follows that $\psi\in\overline{\yoneda_X(X)}$, hence the $\V$-module $\psi:X\modto 1$ has a left adjoint $\varphi:1\modto X$. Furthermore, for any infinite subset $M\subseteq\N$,
\[
 k\le\bigvee_{m\in M}[\psi,x_m^*]\otimes[x_m^*,\psi]
=\bigvee_{m\in M}\varphi(x_m)\otimes\psi(x_m).
\]
Hence, for any $N\in\N$,
\[
\psi(x)
\ge\hom(\bigvee_{n\ge N}\varphi(x_n)\otimes\psi(x_n),\psi(x))
=\bigwedge_{n\ge N}\hom(\varphi(x_n)\otimes\psi(x_n),\psi(x))
\ge\bigwedge_{n\ge N}a(x,x_n),
\]
where the last inequality follows from
\[
a(x,x_n)\otimes\psi(x_n)\otimes\varphi(x_n)\le\psi(x)\otimes\varphi(x_n)\le\psi(x).
\]
Therefore $\psi\ge\psi_s$, and similarly one obtains $\varphi\ge\varphi_s$. Lemma \ref{lem:EqAdjMod} guarantees now $\psi=\psi_s$.
\end{proof}
\begin{corollary}
Assume that $k=\top$ is the top-element of $\V$. Let $s$ be a Cauchy sequence in a $\V$-category $X$ and $x\in X$ with $s\rw x$. Then $x^*=\psi_s$.
\end{corollary}
\begin{proof}
From $s\rw x$ it follows that $\widetilde{s}\rw x^*$ in $[X^\op,\V]$, where $\widetilde{s}:=(x_n^*)_{n\in\N}$ and $s=(x_n)_{n\in\N}$, and therefore $x^*=\psi_s$ by the proposition above.
\end{proof}

Hence, under the assumption that $k=\top$ in $\V$, for a Cauchy sequence $s=(x_n)_{n\in\N}$ in a $\V$-category $X$ and $x\in X$ one has:
\begin{align*}
s\rw x &\iff x_*=\varphi_s\;\iff x^*=\psi_s\\
&\iff x_*\le\varphi_s\;\und\; x^*\le\psi_s\\
&\iff \forall y\in X\,.\,\left(\left(a(x,y)\le\bigvee_{N\in\N}\bigwedge_{n\ge N}a(x_n,y)\right)\;\und\; \left(a(y,x)\le\bigvee_{N\in\N}\bigwedge_{n\ge N}a(y,x_n)\right)\right)\\
&\iff \left(k\le\bigvee_{N\in\N}\bigwedge_{n\ge N}a(x_n,x)\right)\;\und\; \left(k\le\bigvee_{N\in\N}\bigwedge_{n\ge N}a(x,x_n)\right)\\
\intertext{(if, moreover, $k=\bigvee\{u\in\V\mid u\ll k\}$)}
&\iff \forall u\ll k\,\exists N\in\N\,\forall n\ge N\,.\,(u\le a(x_n,x)\,\und\,u\le a(x,x_n)).
\end{align*}

\begin{theorem}\label{thm:ModuleVsSequence}
Assume that $k=\top$ in $\V$ and that there is a sequence $(u_n)_{n\in\N}$ in $\V$ satisfying
\begin{enumerate}
\item $\displaystyle{\bigvee_{n\in\N}u_n= k}$,
\item for all $n\in\N$, $u_n\ll k$,
\item for all $n\in\N$, $u_n\le u_{n+1}$.
\end{enumerate}
Then every adjunction $(\varphi:1\modto X)\dashv(\psi:X\modto 1)$ in $\Mod{\V}$ is of the form $\varphi_s\dashv\psi_s$, for some Cauchy sequence $s$ in $X$. Hence, under these conditions, a $\V$-category $X$ is Cauchy complete if and only if every Cauchy sequence converges.
\end{theorem}
\begin{proof}
We set up a sequence $s=(x_n)_{n\in\N}$ in $X$ putting, for each $n \in \N$, $x_n$ so that $u_n\le \varphi(x_n) \otimes \psi(x_n)$. Then $(x_n)_{n\in\N}$ is Cauchy since
\[
\bigvee_{N\in\N}\bigwedge_{n,m\ge N}a(x_n,x_m)
\ge\bigvee_{N\in\N}\bigwedge_{n,m\ge N}\psi(x_n)\otimes\varphi(x_m)
\ge\bigvee_{N\in\N}\bigwedge_{n,m\ge N}u_n\otimes u_m
\ge\bigvee_{N\in\N}u_N\otimes u_N\ge k.
\]
Furthermore, for every $x\in X$ and $N\in\N$,
\[
\bigwedge_{n\ge N}a(x_n,x)
\le\bigwedge_{n\ge N}\hom(\varphi(x_n),\varphi(x))
\le\hom(\bigvee_{n\ge N}u_n,\varphi(x))=\varphi(x),
\]
hence $\varphi_s\le\varphi$. Similarly, $\psi_s\le\psi$, and Lemma \ref{lem:EqAdjMod} implies $\varphi_s=\varphi$ and $\psi_s=\psi$.
\end{proof}

\begin{corollary}\label{cor:CCsym}
Under the assumptions of Theorem \ref{thm:ModuleVsSequence}, if $X$ is a symmetric $\V$-category and $(\varphi:1\modto X)\dashv(\psi:X\modto 1)$, then $\varphi(x)=\psi(x)$ for all $x\in X$. Consequently, for a $\V$-subcategory $M$ of a $\V$-category $Y$, if $M$ is symmetric, then $\overline{M}$ is symmetric too. In particular, 
the Cauchy-completion of a symmetric $\V$-category is again symmetric.
\end{corollary}
\begin{proof}
If $X=(X,a)$ is symmetric and $\varphi\dashv\psi$, then
\[
 \varphi(x)=\bigvee_{N\in\N}\bigwedge_{n\ge N}a(x_n,x)
=\bigvee_{N\in\N}\bigwedge_{n\ge N}a(x,x_n)=\psi(x),
\]
for all $x\in X$. Hence, if $y\in\overline{M}$ and $M\subseteq Y$ is a symmetric sub-$\V$-category of $Y=(Y,b)$, then $b(y,x)=b(x,y)$ for all $x\in M$. If also $y'\in\overline{M}$, then
\[
 b(y,y')=\bigvee_{x\in M}b(y,x)\otimes b(x,y')
=\bigvee_{x\in M}b(y',x)\otimes b(x,y)=b(y',y).\qedhere
\]
\end{proof}

\begin{remark}
We also note that, assuming that $A=\{x\in\V\mid x\ll k\}$ is directed and $k=\bigvee A$, then $\V$ satifies the condition
\[
 k\le u\vee w\,\Rw\,k\le u\text{ or }k\le v,
\]
for all $u,v\in\V$, as well as $\displaystyle{k=\bigvee_{x\ll k}x\otimes x}$ (for the latter, see \cite[Theorem 1.12]{Flagg_ComplCont}). For the former, note that $A=A_u\cup A_v$ where $A_u=\{x\in A\mid x\le u\}$ and $A_v=\{x\in A\mid x\le v\}$, and also that directedness of $A$ implies that $k=\bigvee\{x\in A\mid x\ge y\}$, for any $y\in A$. Hence, if $k=\bigvee A_u$, then $k\le u$, otherwise there is some $y\in A$ with $y\notin A_u$. Therefore $\{x\in A\mid x\ge y\}\subseteq A_v$, and we conclude $k\le\bigvee A_v\le v$. Consequently, if $A$ is directed, then the closure $\overline{(-)}$ is topological (see Proposition \ref{prop:PropertiesClosure}). Under the conditions of Theorem \ref{thm:ModuleVsSequence}, this topology is determined by its convergent (Cauchy) sequences in the sense that
\begin{align*}
x\in\overline{M} &\iff \text{there is some sequence $s$ in $M$ with $s\rw x$}\\
&\iff\forall u\ll k\,\exists y\in M\,.\,u\le a(x,y)\,\und\, u\le a(y,x)\\
&\iff\forall u\ll k\,\exists y\in M\,.\,u\ll a(x,y)\,\und\, u\ll a(y,x)\\
&\iff M\cap B(x,u)\cap B(u,x)\neq\varnothing,
\end{align*}
where $B(x,u)=\{y\in X\mid u\ll a(x,y)\}$ and $B(u,x)=\{y\in X\mid u\ll a(y,x)\}$. In fact, the collection of all sets $B_u(x)=B(x,u)\cap B(u,x)$ ($x\in X$, $u\ll k$) is a basis for the topology on $X$. Hence, under these assumptions, the topology considered here coincides with the one in \cite{Flagg_ComplCont}.
\end{remark}

\subsection{Example: probabilistic metric spaces}\label{subsect:ExProbMetConv}

We show now that the notions (and results) of Cauchy completeness for a generic $\V$-category specialise to established concepts for probabilistic metric spaces. Recall from Subsection \ref{subsect:Delta} that the quantale $\Delta$ is completely distributive, the neutral element $\eps$ is the top element of $\Delta$ and one has
\begin{align*}
f_{\delta,u}\ll\eps\quad\text{for all $\delta>0,u<1$}, &&
\eps=\bigvee\{f_{\delta,u}\mid \delta>0,u<1\}=\bigvee_{n\in\N} f_{\frac{1}{n},1-\frac{1}{n}},
\end{align*}
hence $\Delta$ satisfies the conditions of Theorem \ref{thm:ModuleVsSequence}. By definition, a sequence $s=(x_n)_{n\in\N}$ in a probabilistic metric space $X=(X,a)$ is Cauchy if and only if
\[
\forall\delta>0\,\forall u<1\,\exists N\in\N\,\forall n,m\ge N\,\forall t>\delta
\,.\,a(x_n,x_m)(t)\ge u,
\]
which is equivalent to
\[
\forall\delta>0\,\forall u<1\,\exists N\in\N\,\forall n,m\ge N\,.\,a(x_n,x_m)(\delta)\ge u.
\]
Similarly, $s$ converges to $x\in X$ precisely when
\[
\forall\delta>0\,\forall u<1\,\exists N\in\N\,\forall n\ge N\,.\,a(x_n,x)(\delta)\ge u\,\und\,a(x,x_n)(\delta)\ge u,
\]
equivalently, whenever
\begin{align*}
\lim_{n\to\infty}a(x_n,x)(\delta)&\rw 1 &\text{and}&&
\lim_{n\to\infty}a(x,x_n)(\delta)&\rw 1
\end{align*}
for all $\delta>0$. The induced topology on $X$ has the sets
\[
 B_{\delta,u}(x)=\{y\in X\mid a(x,y)(\delta)>u\,\und\,a(y,x)(\delta)>u\}
\]
has basic open sets. 

As for metric spaces, a probabilistic metric space $X$ is Cauchy complete (in the sense that every adjoint pair of $\Delta$-modules is representable) if and only if every Cauchy sequence converges, and a Cauchy completion of $X$ is given by $\yoneda_X:X\to\widetilde{X}$, where $\widetilde{X}$ is the subspace of $\Delta^{X^\op}$ defined by all right adjoint $\Delta$-modules. Furthermore, by Proposition \ref{prop:PreservCauchyViaPhi} and Corollary \ref{cor:XtoFXCauchy}, a metric space $X$ is Cauchy complete in $\MET$ if and only if $I_\infty X$ is Cauchy complete in $\PROBMET$, and Corollary \ref{cor:PreservCauchyInj} and Lemma \ref{lem:PreservFullyDense} imply that $P_\infty:\PROBMET\to\MET$ preserves Cauchy completeness.

Finally, we remark that the Cauchy completion $\widetilde{X}$ of a symmetric and finitary probabilistic metric space $X=(X,a)$ is symmetric and finitary as well. In fact, symmetry of $\widetilde{X}$ follows from Corollary \ref{cor:CCsym}. To see that $\widetilde{X}$ is also finitary, we show first that $\psi(x)\in\Delta$ is finite, for every $x\in X$ and every right adjoint $\Delta$-module $\psi:X\modto 1$. Let $x\in X$, and put $v=\psi(x)(\infty)$. Since $\eps=\bigvee_{x'\in X}\psi(x')$, for every $u<1$ there is some $x'\in X$ with $f_{1,u}\le\psi(x')$. Hence,
\[
 1=a(x,x')(\infty)=\hom(\psi(x'),\psi(x))(\infty)
\le\hom(f_{1,u},f_{0,v})=f_{0,v\trundiv u}(\infty)=v\trundiv u
\]
for all $u<1$, which implies $v=1$. Given now also $\psi:X\modto 1$ in $\widetilde{X}$, then the distance
\[
 [\psi,\psi']=\bigvee_{x\in X}[\psi,x^*]\otimes[x^*,\psi']
=\bigvee_{x\in X}\psi(x)\otimes\psi'(x)
\]
is finite. 

\begin{remark}
The first construction of a Cauchy-completion of a probabilistic metric space was given by Sherwood \cite{She66} by putting an appropriate probabilistic metric on the set of equivalence classes of Cauchy sequences of a given space. 
\end{remark}

\section{Injective and exponentiable $\V$-categories}

\subsection{Exponentiable $\V$-categories}

We recall that an object $X$ in a category $\catfont{C}$ with finite products is \emph{exponentiable} whenever the functor $X\times-:\catfont{C}\to\catfont{C}$ has a right adjoint $(-)^X:\catfont{C}\to\catfont{C}$. Unwinding the definition, such a right adjoint must produce, for each object $Z$ in $\catfont{C}$, an object $Z^X$ in $\catfont{C}$ so that, for all objects $Y$ in $\catfont{C}$, there is a natural bijection
\[
 \catfont{C}(X\times Y,Z)\simeq\catfont{C}(Y,Z^X).
\]
The category $\catfont{C}$ is called Cartesian closed if every object $X$ of $\catfont{C}$ is exponentiable. 

We are interested in the case $\catfont{C}=\Cat{\V}$, where we now also assume that $\V$, seen as a category, is Cartesian closed. The latter just means that the underlying lattice of our quantale $\V$ is Heyting, that is, for all $u,w\in\V$ there is $u\rw w\in\V$ satisfying
\[
 u\wedge v\le w\iff v\le(u\rw w),
\]
for all $v\in\V$. Since $\V$ is complete, $\V$ is Heyting if and only if $\V$ satisfies the frame law:
\[
 u\wedge\bigvee_{i\in I}u_i=\bigvee_{i\in I}u\wedge u_i,
\]
for all $u,u_i\in\V$ ($i\in I$). Hence, if $\V$ is completely distributive, then it is also Heyting.

Let now $X=(X,a)$ be an exponentiable $\V$-category. We can choose $Y=E=(1,k)$ above, and conclude that the underlying set of the $\V$-category $Z^X$ is given by the set of all $\V$-functors of type $X\times E\to Z$. Here we can identify $X\times E$ with the $\V$-category $\widehat{X}=(X,\widehat{a})$ whose underlying set is $X$ and where
\[
 \widehat{a}(x,y)=a(x,y)\wedge k,
\]
for all $x,y\in X$. Hence, if $k=\top$ is the top-element of $\V$, then $Z^X$ is indeed given by the set of all $\V$-functors from $X$ to $Z$. The $\V$-categorical structure $d$ in $Z^X$ is the largest one on $\Cat{V}(\widehat{X},Z)$ making the evaluation map
\[
 \ev:X\times\Cat{V}(\widehat{X},Z)\to Z,\,(x,h)\mapsto h(x)
\]
a $\V$-functor, that is,
\[
 d(h,l)=\bigwedge_{x_1,x_2\in X}(a(x_1,x_2)\rw c(h(x_1),l(x_2))),
\]
for all $\V$-functors $h,l:\widehat{X}\to Z$ where $Z=(Z,c)$. In fact, an arbitrary $\V$-category $X$ is exponentiable if and only if the structure $d$ defined above turns $\Cat{V}(\widehat{X},Z)$ into a $\V$-category. The following characterisation of exponentiable $\V$-categories can be found in \cite{CH_ExpVCat} (see also \cite{CHS_Quantaloid}).

\begin{theorem}\label{thm:CharExpVCats}
If $\V$ is Heyting, then a $\V$-category $(X,a)$ is exponentiable in $\Cat{\V}$ if and only if
\[
\bigvee_{x\in X}
(a(x_0,x)\wedge v_0)\otimes (a(x,x_2)\wedge v_1)\geq
a(x_0,x_2)\wedge (v_0\otimes v_1),
\]
for all $x_0,x_2\in X$ and all $v_0,v_1\in\V$.
\end{theorem}

As also shown in \cite{CH_ExpVCat}, this condition simplifies considerably for metric spaces.

\begin{corollary}
An object $(X,a)$ of $\MET$ is exponentiable in $\MET$ if and only if, for all $x_0,x_2\in X$ with $a(x_0,x_2)<\infty$ and all $u_0,u_1\in[0,\infty]$ with $u_0+u_1=a(x_0,x_2)$,
\[
\forall
\varepsilon>0\; \exists x_1\in
X\;:\;a(x_0,x_1)<u_0+\varepsilon\mbox{ and
}a(x_1,x_2)<u_1+\varepsilon.
\]
\end{corollary}

\subsection{Injectives are exponentiable}

We already observed in \cite{Hof_DualityDistSp} that the criterion above implies that every injective metric space is exponentiable. In general, a $\V$-category $X$ is called \emph{injective} if, for every fully faithful $\V$-functor $i:A\to B$ and every $f:A\to X$ in $\Cat{\V}$, there exists a $\V$-functor $g:B\to X$ with $g\cdot i\simeq f$. Note that the $\V$-category $\V$ is injective in $\Cat{\V}$, for any quantale $\V$.

\begin{theorem}
Assume that $\V$ is Heyting. Then the following assertions are equivalent.
\begin{eqcond}
\item\label{cond:InjExp1} For all $u,v,w\in\V$, $w\wedge(u\otimes v)=\bigvee\{u'\otimes v'\mid u'\le u, v'\le v,u'\otimes v'\le w\}$.
\item\label{cond:InjExp2} Every injective $\V$-category is exponentiable.
\item\label{cond:InjExp3} The $\V$-category $\V$ is exponentiable.
\end{eqcond}
\end{theorem}
\begin{proof}
Assume first that $\V$ satisfies the condition in \eqref{cond:InjExp1}, and let $X=(X,a)$ be an injective $\V$-category. We verify that $X$ satisfies the condition of Theorem \ref{thm:CharExpVCats}. To this end, let $x_0,x_2\in X$ and $u,v\in\V$, and put $w=a(x_0,x_2)$. Let $u'\le u$, $v'\le v$ with $u'\otimes v'\le w$. We define $\V$-categories $A=(A,c)$ and $B=(B,d)$ where $A=\{0,2\}$, $B=\{0,1,2\}$ and
\begin{align*}
 c(0,2)=d(0,2)=u'\otimes v',&& d(0,1)=u', && d(1,2)=v',
\end{align*}
and $c(x,x)=d(x,x)=k$ and $c(x,y)=d(x,y)=\bot$ in all other cases. Since $X$ is injective, there is an extension $g:B\to X$ of the $\V$-functor $f:A\to X,\,i\mapsto x_i$ along the inclusion $\V$-functor $A\hrw B$,
\[
\xymatrix{\{0\xrightarrow{\,u'\otimes w'\,}2\}\ar@{^(->}[r]\ar[dr]_f &
\{0\xrightarrow{\,u'}1\xrightarrow{\,v'}2\}\ar[d]^g\\
&X}
\]
and for $x:=g(1)$ one has $a(x_0,x)\geq u'$ and $a(x,x_2)\geq v'$, hence
\[
 (a(x_0,x)\wedge u)\otimes(a(x,x_2)\wedge v)\geq u'\otimes v'.
\]
Consequently,
\[
\bigvee_{x\in X}a(x_0,x)\wedge u)\otimes(a(x,x_2)\wedge v)\geq
\bigvee\{u'\otimes v'\mid u'\le u, v'\le v,u'\otimes v'\le w\}=a(x_0,x_2)\wedge(u\otimes v),
\]
that is, $X$ is exponentiable. Certainly, \eqref{cond:InjExp2} implies \eqref{cond:InjExp3}. Assume now that $\V$ is exponentiable. Hence, for $u,v,w\in\V$, 
\[
\bigvee\{(\hom(k,x)\wedge u)\otimes(\hom(x,w)\wedge v)\mid x\in\V\}
=\hom(k,w)\wedge(u\otimes v)=w\wedge(u\otimes v),
\]
and the condition of \eqref{cond:InjExp1} follows since $(\hom(k,x)\wedge u)\otimes(\hom(x,w)\wedge v)\le w$, $\hom(k,x)\wedge u\le u$ and $\hom(x,w)\wedge v\le v$.
\end{proof}

\begin{corollary}
Assume that $\V$ satisfies $w\wedge(u\otimes v)=\bigvee\{u'\otimes v'\mid u'\le u, v'\le v,u'\otimes v'\le w\}$, for all $u,v,w\in\V$. Then the full subcategory of $\Cat{\V}$ defined by the injective $\V$-categories is Cartesian closed.
\end{corollary}
\begin{proof}
Clearly, any product of injective $\V$-categories is injective. Furthermore, a categorical standard argument (see \cite[Lemma 4.10]{Joh_StoneSp}, or the proof of Corollary \ref{cor:PreservCauchyInj}) shows that with $Y$ also $Y^X$ is injective, for every exponentiable $\V$-category $X$.
\end{proof}

Given $u,v,w \in [0,\infty]$, there are always $u', v' \in [0,+\infty]$ such that $w \wedge(u \otimes v)= u' \otimes v'$, hence every injective metric space is exponentiable. More generally, the same argument shows that every injective $\V$-category is exponentiable, for $\V$ being linearly ordered.

\subsection{Example: probabilistic metric spaces}

The quantale $\Delta$ is completely distributive, hence in particular Heyting. To verify that $w\wedge(u\otimes v)=\bigvee\{u'\otimes v'\mid u'\le u, v'\le v,u'\otimes v'\le w\}$, for all $u,v,w\in\Delta$, it is enough to consider $u=f_{\delta_1,\alpha_1}$, $v=f_{\delta_2,\alpha_2}$ and $w=f_{\delta_3,\alpha_3}$, then $u\otimes v= f_{\delta_1+\delta_2,\alpha_1\cdot\alpha_2}$. We discuss the various possibilities:
\begin{itemize}
\item if $\delta_1+\delta_2\ge\delta_3$ and $\alpha_1\cdot\alpha_2\le\alpha_3$, then $w\wedge(u\otimes v)=(u\otimes v)$;
\item if $\delta_1+\delta_2\ge\delta_3$ and $\alpha_1\cdot\alpha_2\ge\alpha_3$, then $w\wedge(u\otimes v)=(u'\otimes v)$ where $u'=f_{\delta_1,\alpha}$ with $\alpha\cdot\alpha_2=\alpha_3$, $\alpha\le\alpha_1$;
\item if $\delta_1+\delta_2\le\delta_3$ and $\alpha_1\cdot\alpha_2\le\alpha_3$, then $w\wedge(u\otimes v)=(u'\otimes v)$ where $u'=f_{\delta,\alpha_1}$ with $\delta+\delta_2=\delta_3$, $\delta_1\le\delta$;
\item if $\delta_1+\delta_2\le\delta_3$ and $\alpha_1\cdot\alpha_2\ge\alpha_3$, then $w\wedge(u\otimes v)=(u'\otimes v)$ where $u'=f_{\delta,\alpha}$ with $\delta+\delta_2=\delta_3$, $\delta\ge\delta_1$ and $\alpha\cdot\alpha_2=\alpha_3$, $\alpha\le\alpha_1$.
\end{itemize}
Therefore one has
\begin{theorem}
Every injective probabilistic metric space is exponentiable, in particular, $\Delta$ is exponentiable in $\PROBMET$. Furthermore, the full subcategory of $\PROBMET$ defined by all injective probabilistic metric spaces is Cartesian closed.
\end{theorem}


\begin{thebibliography}{10}

\bibitem{AHS}
{\sc J.~Ad{\'a}mek, H.~Herrlich, and G.~E. Strecker}, {\em Abstract and
  concrete categories}, Pure and Applied Mathematics (New York), John Wiley \&
  Sons Inc., New York, 1990.
\newblock {T}he joy of cats, A Wiley-Interscience Publication.

\bibitem{BBR_GenMet}
{\sc M.~M. Bonsangue, F.~{van Breugel}, and J.~J. M.~M. Rutten}, {\em
  Generalized metric spaces: completion, topology, and powerdomains via the
  {Y}oneda embedding}, Theoret. Comput. Sci., 193 (1998), pp.~1--51.

\bibitem{CH_ExpVCat}
{\sc M.~M. Clementino and D.~Hofmann}, {\em Exponentiation in
  {$V$}-categories}, Topology Appl., 153 (2006), pp.~3113--3128.

\bibitem{CHS_Quantaloid}
{\sc M.~M. Clementino, D.~Hofmann, and I.~Stubbe}, {\em Exponentiable functors
  between quantaloid-enriched categories}, Appl. Categ. Structures, 17 (2009),
  pp.~91--101, arXiv:math.CT/0604569.

\bibitem{EK_CloCat}
{\sc S.~Eilenberg and G.~M. Kelly}, {\em Closed categories}, in Proc. Conf.
  Categorical Algebra (La Jolla, Calif., 1965), Springer, New York, 1966,
  pp.~421--562.

\bibitem{Flagg_ComplCont}
{\sc R.~C. Flagg}, {\em {Completeness in continuity spaces.}}
\newblock {Seely, R. A. G. (ed.), Category theory 1991. Proceedings of an
  international summer category theory meeting, held in Montr{\'e}al,
  Qu{\'e}bec, Canada, June 23-30, 1991. Providence, RI: American Mathematical
  Society. CMS Conf. Proc. 13, 183-199 (1992).}, 1992.

\bibitem{Flagg_QuantCont}
\leavevmode\vrule height 2pt depth -1.6pt width 23pt, {\em Quantales and
  continuity spaces}, Algebra Universalis, 37 (1997), pp.~257--276.

\bibitem{FK_ContSp}
{\sc R.~C. Flagg and R.~Kopperman}, {\em Continuity spaces: reconciling domains
  and metric spaces}, Theoret. Comput. Sci., 177 (1997), pp.~111--138.
\newblock Mathematical foundations of programming semantics (Manhattan, KS,
  1994).

\bibitem{FSW_QDT}
{\sc R.~C. Flagg, P.~S{\"u}nderhauf, and K.~Wagner}, {\em A logical approach to
  quantitative domain theory}, Topology Atlas Preprint \# 23,  (1996),
  http://at.yorku.ca/e/a/p/p/23.htm.

\bibitem{GV94}
{\sc A.~George and P.~Veeramani}, {\em On some results in fuzzy metric spaces},
  Fuzzy Sets and Systems, 64 (1994), pp.~395--399.

\bibitem{GR02}
{\sc V.~Gregori and S.~Romaguera}, {\em On completion of fuzzy metric spaces},
  Fuzzy Sets and Systems, 130 (2002), pp.~399--404.
\newblock Theme: Fuzzy intervals.

\bibitem{Hof_DualityDistSp}
{\sc D.~Hofmann}, {\em Duality for distributive spaces}, tech. rep., 2010,
  arXiv:math.CT/1009.3892.

\bibitem{HT_LCls}
{\sc D.~Hofmann and W.~Tholen}, {\em {L}awvere completion and separation via
  closure}, Appl. Categ. Structures, 18 (2010), pp.~259--287,
  arXiv:math.CT/0801.0199.

\bibitem{Joh_StoneSp}
{\sc P.~T. Johnstone}, {\em Stone spaces}, vol.~3 of Cambridge Studies in
  Advanced Mathematics, Cambridge University Press, Cambridge, 1982.

\bibitem{Kel_EnrCat}
{\sc G.~M. Kelly}, {\em Basic concepts of enriched category theory}, vol.~64 of
  London Mathematical Society Lecture Note Series, Cambridge University Press,
  Cambridge, 1982.
\newblock Also in: Repr. Theory Appl. Categ. No.~10 (2005), 1--136.

\bibitem{Kop88}
{\sc R.~Kopperman}, {\em All topologies come from generalized metrics}, Amer.
  Math. Monthly, 95 (1988), pp.~89--97.

\bibitem{Law_MetLogClo}
{\sc F.~W. Lawvere}, {\em Metric spaces, generalized logic, and closed
  categories}, Rend. Sem. Mat. Fis. Milano, 43 (1973), pp.~135--166 (1974).
\newblock Also in: Repr. Theory Appl. Categ. No.~1 (2002), 1--37.

\bibitem{MacLane_WorkMath}
{\sc S.~MacLane}, {\em Categories for the working mathematician},
  Springer-Verlag, New York, 1971.
\newblock Graduate Texts in Mathematics, Vol. 5.

\bibitem{Men42}
{\sc K.~Menger}, {\em Statistical metrics}, Proc. Nat. Acad. Sci. U. S. A., 28
  (1942), pp.~535--537.

\bibitem{Raney_CD}
{\sc G.~N. Raney}, {\em Completely distributive complete lattices}, Proc. Amer.
  Math. Soc., 3 (1952), pp.~677--680.

\bibitem{Rut98a}
{\sc J.~J. M.~M. Rutten}, {\em Weighted colimits and formal balls in
  generalized metric spaces}, Topology Appl., 89 (1998), pp.~179--202.
\newblock Domain theory.

\bibitem{SS83}
{\sc B.~Schweizer and A.~Sklar}, {\em Probabilistic metric spaces},
  North-Holland Series in Probability and Applied Mathematics, North-Holland
  Publishing Co., New York, 1983.

\bibitem{She66}
{\sc H.~Sherwood}, {\em On the completion of probabilistic metric spaces}, Z.
  Wahrscheinlichkeitstheorie und Verw. Gebiete, 6 (1966), pp.~62--64.

\bibitem{Wag_PhD}
{\sc K.~R. Wagner}, {\em Solving Recursive Domain Equations with Enriched
  Categories}, PhD thesis, Carnegie Mellon University, 1994,
  ftp://ftp.risc.uni-linz.ac.at/pub/techreports/1994/94-62.ps.gz.

\bibitem{Woo_OrdAdj}
{\sc R.~Wood}, {\em Ordered sets via adjunction}, in Categorical foundations,
  vol.~97 of Encyclopedia Math. Appl., Cambridge Univ. Press, Cambridge, 2004,
  pp.~5--47.

\end{thebibliography}

\end{document}